\documentclass[
a4paper,
11pt,
twoside,
]{article}
\usepackage[utf8]{inputenc}
\usepackage[T1]{fontenc}
\usepackage[sc]{mathpazo}
\usepackage[english]{babel}
\usepackage{geometry}
\usepackage{amsmath,amsfonts,amssymb,amsthm}
\usepackage[final,colorlinks]{hyperref}
\usepackage{graphicx,subcaption}
\usepackage{mathtools}
\usepackage[dvipsnames]{xcolor}
\usepackage{url}
\definecolor{darkblue}{cmyk}{1,0,0,0.8}
\definecolor{darkred}{cmyk}{0,1,0,0.7}
\hypersetup{anchorcolor=black,
  citecolor=darkblue, filecolor=darkblue,
  menucolor=darkblue,urlcolor=darkblue,linkcolor=darkblue}
\usepackage{pdflscape}
\usepackage{etoolbox,paralist}
\usepackage{fancyhdr}
\usepackage{comment}
\usepackage{lastpage}
\usepackage{ifdraft}
\usepackage{hyphenat}
\usepackage{natbib}
\ifdraft{\usepackage{showlabels}}{}
\usepackage{enumitem}
\usepackage[multiple]{footmisc}
\usepackage{siunitx}
\usepackage[mathlines]{lineno}

\newcommand*\patchAmsMathEnvironmentForLineno[1]{%
  \expandafter\let\csname old#1\expandafter\endcsname\csname #1\endcsname
  \expandafter\let\csname oldend#1\expandafter\endcsname\csname end#1\endcsname
  \renewenvironment{#1}%
  {\linenomath\csname old#1\endcsname}%
  {\csname oldend#1\endcsname\endlinenomath}}%
\newcommand*\patchBothAmsMathEnvironmentsForLineno[1]{%
  \patchAmsMathEnvironmentForLineno{#1}%
  \patchAmsMathEnvironmentForLineno{#1*}}%
\patchBothAmsMathEnvironmentsForLineno{equation}%
\patchBothAmsMathEnvironmentsForLineno{align}%
\patchBothAmsMathEnvironmentsForLineno{flalign}%
\patchBothAmsMathEnvironmentsForLineno{alignat}%
\patchBothAmsMathEnvironmentsForLineno{gather}%
\patchBothAmsMathEnvironmentsForLineno{multline}%

\DeclareMathOperator{\lip}{Lip}

\DeclareMathOperator{\rg}{rg}
\DeclareMathOperator{\Lin}{Lin}

\newcommand{\dd}{\mathop{}\!\mathrm{d}}
\newcommand{\lm}{{\ell_{\max}}}
\newcommand{\gts}{{G_\mathrm{ts}}}

\theoremstyle{plain}
\newtheorem{theorem}{Theorem}[section]
\newtheorem{lemma}[theorem]{Lemma}
\newtheorem{proposition}[theorem]{Proposition}
\newtheorem{corollary}[theorem]{Corollary}
\newtheorem{definition}[theorem]{Definition}
\newtheorem{assumption}[theorem]{Assumption}

\newtheorem{remark}[theorem]{Remark}

\numberwithin{equation}{section}
\numberwithin{table}{section}
\numberwithin{figure}{section}

\newcommand{\R}{\mathbb{R}}
\newcommand{\Z}{\mathbb{Z}}

\newcommand{\gfde}{G_\mathrm{FDE}}

\fancypagestyle{mypagestyle}{%
  \fancyhf{}%
  \fancyhead[LO,RE]{\scriptsize Boundary-value problems with state-dependent delays}%
  \fancyhead[RO,LE]{\scriptsize }%
  \fancyfoot[LE,RO]{\scriptsize \thepage\,/\,\pageref*{LastPage}}%
}
\title{Boundary-value problems of functional differential equations with state-dependent delays}
\author{
Alessia And\`o$^{1}$, Jan Sieber$^{2}$
\\[.5em]
\small $^{1}$CDLab -- Computational Dynamics Laboratory\\[-.2em]
\small Department of Mathematics, Computer Science and Physics -- University of Udine\\[-.2em]
\small via delle scienze 206, 33100 Udine, Italy\\[-.2em]
\small \texttt{alessia.ando@uniud.it}\\[.5em]
\small $^{2}$Department of Mathematics and Statistics -- University of Exeter\\[-.2em]
\small North Park Road, Exeter EX4 4QF, U.\,K.\\[-.2em]
\small \texttt{j.sieber@exeter.ac.uk}
}
\date{\today}
\begin{document}
\clearpage
\maketitle
\thispagestyle{empty}
\begin{abstract}\noindent
  We prove convergence of piecewise polynomial collocation methods
  applied to periodic boundary value problems for functional
  differential equations with state-dependent delays. The state
  dependence of the delays leads to nonlinearities that are not
  locally Lipschitz continuous preventing the direct application of
  general abstract discretization theoretic frameworks. We employ a
  weaker form of differentiability, which we call mild
  differentiability, to prove that a locally unique solution of the
  functional differential equation is approximated by the solution of
  the discretized problem with the expected order.
\end{abstract}


\section{Introduction}
Tracking time-periodic responses (\emph{periodic orbits}) is a common
task in numerical bifurcation analysis of nonlinear dynamical systems
\citep{K04,GK07}. The interest extends to periodic orbits that are
dynamically unstable or extremely sensitive to parameters, as these
orbits are thresholds between alternative stable states or are the
link between seemingly discontinuous system responses, such as canards
\citep{desroches2012mixed}. Dynamically unstable or sensitive orbits
are beyond the reach of initial-value problem (IVP) solvers due to
sensitivity to initial conditions near the periodic orbit. If the
dynamical system is described by ordinary differential equations
(ODEs) there are robust tools available that address sensitivity
caused by dynamical instability and are able to track families of
periodic orbits in one or many system parameters. Widely adopted tools
for ODEs are \textsc{AUTO} \citep{D07}, \textsc{MatCont} \citep{GK07},
or \textsc{coco} \citep{DS13}. These tools compute solutions
$y(t)\in\R^{n_y}$ of an ODE with parameters, $y'(t)=G(y(t),p)$ with an
unknown period $T$, such that $y(t)=y(t+T)$ for all $t\in\R$ and some
unknown $T>0$. They solve the boundary-value problem (BVP) numerically
with piecewise polynomial collocation similar to that described by \citet{ascher1981collocation}, where the time-rescaled periodic orbit
$t\mapsto y(t/T)$ is approximated by a piecewise polynomial
$s\mapsto y^L(s)$ for $s\in[0,1]$ with $L$ pieces of (usually uniform)
degree $m$ on $L$ subintervals of $[0,1]$. The method imposes the ODE
at chosen time points (nodes) within each subinterval to construct a large nonlinear system of algebraic equations with a
blockdiagonal Jacobian. Tools such as \textsc{AUTO}, \textsc{MatCont}
and \textsc{coco} \emph{embed} the BVP by including one or several
parameters into the unknown and augmenting the BVP with constraints
(typically affine), such as phase and pseudo-arclength
conditions. This augments the blockdiagonal Jacobian, resulting in
well-conditioned problems that would be ill-conditioned if a shooting
approach over a bounded number of time intervals was employed
instead. See \citep{desroches2012mixed} for several impressive
demonstrations of how to find phenomena that occur in exponentially small
parameter regions in singularly perturbed problems.

\paragraph{Functional differential equations (FDEs)}
If the right-hand side of the differential equation depends on
values of $y$ at times other than the current $t$, one speaks of
functional differential equations, writing
\begin{align}
  \label{intro:fde}
  \dot y(t)&=\gfde(y_t,p)\mbox{, where\quad}\gfde:C^0([-\tau_{\max},0];\R^{n_y})\times\R^{n_p}\to\R^{n_y}
\end{align}
where the dependent variable is $y:\R\to\R^{n_y}$, $p\in\R^{n_p}$ are
the problem parameters, and
  $\gfde$ 
is a continuous nonlinear functional. The most common example class for equations such as \eqref{intro:fde} are \emph{delay differential equations} (DDEs). Our illustrative example \eqref{intro:dde:ex} below is a DDE. We use $C^\ell(I;\R^n)$ for the space
of continuous (for $\ell=0$) or continuously differentiable (for
$\ell>0$) functions from the interval $I\subset \R$ into $\R^n$. The subscript $t$ in $y_t$ denotes the time shift operator
\begin{align*}
  C^0(\R;\R^{n_y})\times\R\ni(y,t)\mapsto y_t\in C^0([-\tau_{\max},0];\R^{n_y})\mbox{\ with\ }y_t(s)=y(t+s)\mbox{}
\end{align*}
(we may also write $\dot y(t)=G_\mathrm{FDE}(y(t+(\cdot)),p)$).
Common challenging bifurcation analysis problems involving FDEs arise in optics due to transport delays \citep{Seidel22}, in population
dynamics due to non-zero maturation times (leading to implicitly
defined threshold delays and integrals over the past
\citep{gedeon2022operon,diekmann2010daphnia}), or in machining due to the effects of
the machining tool on the surface from the previous revolution
\citep{IBS08}.

\paragraph{Collocation for FDEs}
As the piecewise polynomial $y^L$ provides a natural interpolation the
collocation methods for ODEs immediately generalize to FDEs. After
time rescaling, one is looking for a $1$-periodic function satisfying
the FDE
\begin{align}\label{intro:bvp}
  y'(t)=TG_\mathrm{FDE}(y(t+(\cdot)/T),p),
\end{align}
such that one may evaluate the functional $G_\mathrm{FDE}$ at a
collocation point $t$ using the piecewise polynomial:
$(y^L)'(t)=TG_\mathrm{FDE}(y^L(t+(\cdot)/T),p)$. Imposition of the FDE
at a time point $t\in[0,1]$ introduces coupling between values of $y$
at different times such that the Jacobian of the resulting nonlinear
system of algebraic equations is no longer blockdiagonal. When one
seeks to find periodic orbits one may ``wrap around'' when finding the
value of $y$ at deviating arguments $t+s$ outside the base interval
$[0,1]$ by using $(t+s)\mod 1$. So the non-diagonal coupling is the
only added difficulty when formulating discretized periodic BVPs for
FDEs. This motivates specialized BVP solvers and analysis for the case
of finding periodic orbits in FDE problems with parameters.

Complete tools for bifurcation analysis incorporating collocation for
periodic orbits were developed and implemented by \cite{ELR02}
(\textsc{DDE-Biftool}) and \cite{S06c} (\textsc{knut}), see also
\citep{RS07} for a review. These tools permit an arbitrary number of
discrete delays as part of $G_\mathrm{FDE}$, which may depend on the
state and system parameters (for \textsc{DDE-Biftool}). A-posteriori
convergence tests on examples suggested convergence orders equal to
the degree $m$ of the polynomial pieces for the maximum norm of the
error over interval $[0,1]$. One cannot expect better (e.g.,
superconvergence) as the interpolation uses the piecewise polynomial
$y^L(\cdot)$ of degree $m$ when evaluating $G_\mathrm{FDE}$ at the
collocation points \citep{ELHR01,BKW06a}. \cite{ED02} proved linear stability for collocation methods for time-periodic linear
inhomogeneous FDEs with discrete constant delays. They pointed to
``general stability theory for discretizations of nonlinear operator
equations'' for concluding (informally) convergence of the methods,
referring to \citep{K75}.
An alternative are methods based on series expansions of Chebyshev or Fourier type, so, fixing $L=1$, letting $m$ go to infinity, and choosing a projection of the right-hand side onto the truncated series up to $m$. The high-order a-posteriori convergence observed for these methods makes them excellent candidates for verified numerical computations. Gimeno \emph{et al.}\ demonstrated how Chebyshev series truncations can be applied to some DDEs with a single state-dependent delay to prove existence of periodic orbits  and their isochrons with verified numerics \citep{glmy23,gjl21}.

\paragraph{Lack of continuous differentiability of the nonlinearity}
However, the argument by \cite{ED02} is only valid if one treats the
period $T$ of the orbit and the delays (which are system parameters) as
known constants.

Let us illustrate the obstacle to applying standard arguments for
convergence of numerical methods with the simple example
\begin{align}
  \label{intro:dde:ex} \dot y(t)=-y(t-p-y(t)),
\end{align}
where $p\approx \pi/2$.
This FDE fits the general form \eqref{intro:fde} with
  $\gfde(y,p)=-y(-p-y(0))$,
which is continuous on an open subset of
$C^0([-\tau_{\max},0];\R)\times\R$ for $\tau_{\max}>\frac{\pi}{2}$
(so, $n_y=n_p=1$).  The Hopf bifurcation theorem ensures that this FDE
has a family of small-amplitude periodic solutions of the form
$y(t)=y_0\sin(\theta+t/T)+O(|y_0|^2)$ with $0<|y_0|\ll1$ and period
$T=2\pi+O(|y_0|^2)$ for $p=\pi/2+O(|y_0|^2)$ and arbitrary $\theta$.  See
\citep{K04} for the classical Hopf bifurcation theorem,
\citep{H77,DGLW95} for the version for FDEs with constant delay, and \citep{S12} for its
proof for FDEs with state-dependent delays. After rescaling the time
interval to $[0,1]$ the unknown period $T$ appears explicitly as a
parameter we look for $1$-periodic functions $y$,
periods $T$, and parameters $p$ such that
\begin{align}\label{intro:dde:ex:rescaled}
  y'(t)=G(y_t,T,p)=-Ty(t-p/T-y(t)/T)\mbox{ for all $t\in[0,1)$.}
\end{align}
For the rescaled problem \eqref{intro:dde:ex:rescaled}  the right-hand side nonlinearity has the form
\begin{align}\label{intro:Gts:ex}
  \gts:C^\ell_\pi\times\R\times\R\ni (y(\cdot),T,p)\mapsto Ty\left((\cdot)-p/T-y(\cdot)/T\right)\in C^\ell_\pi
\end{align}
where the subscript $(\cdot)_\mathrm{ts}$ stands for \emph{time shift} and we use $C^\ell_\pi$ for spaces of $\ell$ times continuously
differentiable $1$-periodic functions. 
We observe that unknowns appear inside the arguments of $y$, which is itself unknown, such that application of the derivative to $\gts$ reduces the regularity
of the argument $y$ of $\gts$:
\begin{equation}\label{intro:DG:ex}
  \begin{aligned}[t]
    \lefteqn{D \gts(y,T,p)[\delta^y,\delta^T,\delta^p](t)=}\\
    &-T\delta^y(t-p/T-y(t)/T)+y'(t-p/T-y(t)/T)[\delta^y(t)+\delta^p]/T\\
    &-\delta^T \left[y(t-p/T-y(t)/T)+y'(t-p/T-y(t)/T)[p+y(t)]/T\right]
  \end{aligned}
\end{equation}
for a deviation $(\delta^y,\delta^T,\delta^p)\in C^\ell_\pi\times\R\times\R$. Thus, the right-hand side $\gts$ is not differentiable or locally
Lipschitz continuous if we consider $\gts$ as mapping $y\in C^\ell_\pi$
into $C^\ell_\pi$ for any $\ell\geq0$. The review by \cite{HKWW06} pointed out this lack
of continuous differentiability and its consequences (see also \cite{cassidy2019}). For example,
solutions for IVPs are not unique if one permits $C^0$ initial
conditions. If the initial conditions are at least $C^1$ and
\emph{compatible} \citep{W03}, then the dependence of IVP solutions on
initial values is $C^1$ but results on higher regularity
are missing. Large parts of the theory for FDEs as developed in
textbooks by \cite{H77}, \cite{HL93} and \cite{DGLW95} relies on
continuous differentiability of $\gts$ acting on arguments $y\in C^0$,
and are, thus, not applicable to problems such as
\eqref{intro:dde:ex}.

\paragraph{Convergence of numerical discretization methods for BVPs}
Similarly, discretization theory for boundary-value problems (BVPs)
has been developed for nonlinearities $\gts$ that are 
differentiable for $C^0$ arguments by \cite{mas15SINA2,mas15SINA1,mas15NM}. Indeed, it is tempting to assume differentiability of the nonlinearities and their discretized counterparts in FDE BVPs, because theorems on zeros of differentiable operators such as \cite[Lemma 19.1]{kra72} allow one to prove the uniqueness of the fixed point of the discretized problem straightforwardly. \cite{andoSIAM2020} pointed out that for periodic BVPs with unknown period $T$ the rescaling by the unknown
period $T$ introduces a state dependence of the deviating time
arguments (namely on $T$, see the term $Ty(t-p/T\ldots)$ in
\eqref{intro:dde:ex:rescaled}), even for FDE problems with constant
delay. This causes a loss of differentiability of the problem's
nonlinearity with respect to the unknown $T$. Careful reanalysis of
the general framework constructed by \cite{mas15NM} showed that
continuous differentiability with respect to the scalar variable $T$
is only needed in a single point, namely the assumed-to-exist solution
of the BVP. Thus, \cite{andoSIAM2020} proved convergence of the
routinely used methods in \textsc{DDE-Biftool} and \textsc{knut} for
FDEs with constant delays for the first time, closing the gap left in the argument of
\cite{ED02}.

The analysis of \cite{andoSIAM2020} leaves the question open how necessary special treatment of the unknown period $T$ is (a
finite-dimensional part of the unknowns of the problem), or if
convergence of numerical discretization for periodic BVPs can be
proved without assuming continuous differentiability of the right-hand
side (also referred to as \emph{Fr{\'e}chet differentiability} in, e.g., \citep{ampr95})
.
The review by \cite{HKWW06} points to the appropriate generalized
differentiability properties that are satisfied by the nonlinearities
occuring in FDEs. As one can see in the right-hand side
\eqref{intro:Gts:ex} and its
derivative \eqref{intro:DG:ex}, nonlinearities in FDEs are differentiable $\ell$ times with respect to their arguments $y$ and
$p$ only if $y\in C^\ell$. In \eqref{intro:DG:ex} we also
see that, while the derivative depends on $y'$, it does not depend on
$(\delta^y)'$. We use this property of restricted differentiability
(we call it \emph{mild} in Definition~\ref{def:mild}) to prove our
central convergence result for discretizations of embedded periodic
BVPs of FDEs. In particular, our main result is the error estimate $O(L^{-\min\{\ell_{\max},m\}})$ for right-hand sides which are mildly differentiable to order $\ell_{\max}$, where $m$ is the fixed degree of the approximating piecewise polynomials and $L$ is the increasing size of the mesh.

\section{Main results}
\label{sec:result}

\subsection{Periodic BVPs for FDEs and mild differentiability}
\label{sec:res:mild}
Consider an embedded periodic BVP of the form
\begin{align}
  \label{res:bvp}
  y'(t)&=TG_\mathrm{FDE}(y(t+(\cdot)/T),p),&
  0&=R_\mathrm{aff}[y,T,p]\mbox{,}
\end{align}
for continuous $1$-periodic $y:\R\to \R^{n_y}$, period $T$, and
parameter $p\in\R^{n_p}$, where
$G_\mathrm{FDE}:C^0([-\tau_{\max},0];\R^{n_y})\times
\R^{n_p}\to\R^{n_y}$ is continuous. The affine map
$R_\mathrm{aff}:C^0_\pi\times \R^{n_p+1}\to\R^{n_p+1}$ defines
constraints, making the system ``square'' when including the
parameters $T$ and $p$ as unknowns. 
The constraints should also eliminate the
shift symmetry that $y(\theta+(\cdot))$ is a solution of
\eqref{intro:dde:ex} for all $\theta$ whenever it is a solution for
$\theta=0$. 
Nonlinearities such as $G_\mathrm{FDE}$ require the concept of \emph{mild differentiability} to describe their regularity.
\begin{definition}[Mild differentiability]\label{def:mild}
  A functional $G:C^0\to\R^{n_G}$ is called $\lm$ times mildly
  differentiable if
   \begin{enumerate}
   \item \label{def:mild:rest}\textup{\textbf{\textsf{(restricted continuous differentiability)}}}
     its restriction $G\vert_{C^\ell}:C^\ell\to\R^{n_G}$ is $\ell$ times differentiable for $\ell\leq \lm$
     (in particular, $G$ is continuous), and
   \item\label{def:mild:ext} \textup{\textbf{\textsf{(extendability)}}} the map
     $C^\ell\times C^\ell\ni (u,\delta^u)\mapsto D^\ell
     G(u)(\delta^u)^\ell\in\R^{n_G}$ has a continuous extension to
     $C^\ell\times C^{\ell-1}$ for all $\ell\leq \lm$, where $(\delta^u)^\ell$ is the $\ell$-tuple $\delta^u\ldots\delta^u$ of arguments of the multilinear map $D^\ell G(u)$.
   \end{enumerate}
\end{definition}
\noindent
This definition extends naturally to the functional on the right-hand side of \eqref{res:bvp}, 
\begin{align*}
  G(y,T,p):=TG_\mathrm{FDE}(y(t+(\cdot)/T),p),
\end{align*}
through the embedding of $\R^{n_p+1}$ into $C^0_\pi$ that treats a
vector $(T,p)$ as the constant function $t\mapsto (T,p)$. We list basic properties of mildly differentiable functionals in Section~\ref{sec:mild}, which will be needed for our convergence analysis in Section~\ref{sec:convergence}. These include the validity of the chain rule, which ensures that $t\mapsto TG_\mathrm{FDE}(y(t+(\cdot)/T,p)$ is in $C^\ell$ if $y\in C^\ell$ and $G_\mathrm{FDE}$ is $\ell$ times mildly differentiable. The nonlinearity $(y,T,p)\mapsto -Ty(-p/T-y(0)/T)$ used in example
\eqref{intro:dde:ex:rescaled} is an example of a 
mildly differentiable functional (to arbitrary degree). 
\noindent
Definition~\ref{def:mild} extends the classical ``mild regularity assumption'' made to obtain classical results for state-dependent FDEs \citep{W03,HKWW06} to to arbitrary degree $\ell>1$. Alternative approaches taken in \citep{K03,S12} are equivalent but more involved to state and check in practical examples. All problems that can be formulated with standard numerical tools such as \textsc{DDE-Biftool} and \textsc{ddesd} available in MATLAB satisfy mild differentiability in the sense of Definition~\ref{def:mild} if the problem coefficient functions are sufficiently regular. See \citep{HBCHS15} for a demonstration of how \textsc{DDE-Biftool} can be used for FDEs with state-dependent delays and, hence, only mildly differentiable right-hand sides.
 
\begin{assumption}[Assumptions on the BVP]\label{res:ass}\
  \begin{compactenum}
  \item \label{res:ass:mdiff} \textup{(Mild differentiability)} The
    right-hand side $G_\mathrm{FDE}$ in the FDE
    \eqref{res:bvp} is mildly differentiable to order
    $\lm\geq 1$.
  \item \label{res:ass:existence} \textup{(Existence of solution)} BVP \eqref{res:bvp} has a solution $(y^*,T^*,p^*)$.
  \item \label{res:ass:linear:inv} \textup{(Linear well-posedness)}
    The BVP \eqref{res:bvp}, linearized in $(y^*,T^*,p^*)$,
    has only the trivial solution $(\delta^y,\delta^T,\delta^p)=0$.
  \end{compactenum}
\end{assumption}
Corollary~\ref{thm:regularity:fp} below shows that
points~\ref{res:ass:mdiff} and \ref{res:ass:existence} imply that
$y^*\in C^2_\pi$ such that the derivative $DG(y^*,T^*,p^*)$ exists for
the mildly differentiable $G$, which permits us to pose
Assumption~\ref{res:ass},
point~\ref{res:ass:linear:inv}. Section~\ref{sec:assumptions} will
restate Assumption~\ref{res:ass} for a fixed-point
problem equivalent to \eqref{res:bvp}, defined later in \eqref{gen:fixedpoint}.

\subsection{Convergence of solution of discretized BVP}
\label{sec:res:collocation}
For polynomial collocation the unknown function is a $1$-periodic
continuous piecewise polynomial $y^L$ on a mesh given as
$0=t_0<\ldots<t_L=1$. More precisely, $y^L$ is a polynomial of degree $m$ on
$[t_{i-1},t_i]$ for all $i=1,\ldots,L$ in each of its $n_y$
components, $y^L$ is continuous, and $y(t)=y(t+1)$ for all $t$.  Additional unknowns
are the parameters $(T^L,p^L)$, resulting in $n_y mL+n_p+1$ unknowns
overall.  The system of algebraic equations,
\begin{align}\label{res:bvp:disc}
  0&=(y^L)'(t_{i,j})-T^LG_\mathrm{FDE}(y^L(t_{i,j}+(\cdot)/T^L),p^L),&
    0&=R_\mathrm{aff}[y^L,T^L,p^L]
\end{align}
for $1\leq i\leq L$, $1\leq j\leq m$, evaluates the FDE at the
collocation points $t_{i,j}=t_{i-1}+(t_i-t_{i-1})t_{\mathrm{c},j}$,
where the points $t_{\mathrm{c},j}$ are the $m$ collocation points for
degree $m-1$ on the interval $[0,1]$ for a sequence of orthogonal
polynomials (e.g. of Gauss-Legendre or Chebyshev type). The strategy
for adjusting approximation quality is a finite-element approach
\citep{andophd2020,andoSIAM2020}, keeping the degree $m$ bounded, and
considering the limit $L\to\infty$, refining the mesh such that
  $\max\{t_i-t_{i-1}\}\leq C_\mathrm{msh}/L$
for a bounded $C_\mathrm{msh}$ independent of $L$. In contrast, the strategy of keeping $L$ bounded and letting $m$ go to infinity is called the
\emph{spectral} element method in
\citep{breda2005pseudospectral,T96}. Spectral methods promise
exponential convergence, but require bounds on all derivatives of
$y^*$. In the proofs of
Lemmas~\ref{thm:DPhi:consistency:alt} and \ref{lemma:nonlin} we
rely on the boundedness of the degree $m$, but both Lemmas only require
$G$ to be mildly differentiable once.

We can now state a convergence theorem for the discretized
BVP \eqref{res:bvp:disc}.
\begin{theorem}[Convergence of discretization]\label{res:thm:conv}
  Under Assumption~\ref{res:ass} the discretized BVP
  \eqref{res:bvp:disc} with  $\max\{t_i-t_{i-1}\}\leq C_\mathrm{msh}/L$ 
  has a locally
  unique solution $x_L=(y^L,T^L,p^L)$ near $x^*=(y^*,T^*,p^*)$ for all
  sufficiently large $L$, satisfying 
  \begin{align}\label{res:thm:conv:est}
    \|x_L-x^*\|_{0,1}:=\max\{\|y^L-y^*\|_{0,1},|T^L-T^*|,|p^L-p^*|\}=O(L^{-\min\{\lm,m\}}).
  \end{align}
\end{theorem}
\noindent
The norm $\|y^L-y^*\|_{0,1}$ in \eqref{res:thm:conv:est} is the Lipschitz norm of the discretization error.

\paragraph{Proof through equivalent fixed-point problem}
\cite{andophd2020,andoSIAM2020} reformulate the BVP \eqref{res:bvp}
as a fixed-point problem, following the general framework in
\cite{mas15NM}. We modify this approach in
section~\ref{sec:formulation} to exploit the special structure present
in periodic BVPs by constructing a fixed-point problem where the
right-hand side is compact and maps spaces of periodic functions into
spaces of periodic functions.

For this fixed-point problem the discretization corresponds to
inserting a projection operator $\mathcal{P}_L$ onto the space of
discontinuous piecewise polynomials of degree $m-1$ in between the
nonlinearity and an integral operator. This formulation ensures that
the space of numerical solutions $y^L$ consists of continuous
$1$-periodic piecewise polynomials of degree
$m$. Theorem~\ref{thm:conv} in section~\ref{sec:convergence}
establishes convergence for the discretized fixed-point problem.

Mild differentiability of $G_\mathrm{FDE}$ helps us
in section~\ref{sec:convergence} to establish consistence and stability of the numerical method
and the smallness of the nonlinear terms, leading to the proof of
Theorem~\ref{thm:conv} and, hence, Theorem~\ref{res:thm:conv}.

\subsection{Illustrative example of BVP}
\label{sec:res:example}
We will use the example \eqref{intro:dde:ex} throughout to illustrate concepts and results. The embedded BVP for finding periodic orbits of \eqref{intro:dde:ex}, rescaled to base interval $[0,1]$ with  additional affine conditions is
\begin{align}
  y'(t)&=-Ty(t-(p+y(t))/T),\mbox{\quad ($y(t)=y(t+1)$ for all $t\in\R$),}\label{dde:ex1:rescaled}\\
  0&=R_\mathrm{aff}[y,T,p]:=
  \begin{bmatrix}
    y(0)\\
    2\int_0^1\sin(2\pi t)y(t)\dd t-y_0
  \end{bmatrix}\label{Raff:ex1}
\end{align}
for a range of small $y_0$. It is sufficient to impose
\eqref{dde:ex1:rescaled} on the base interval $[0,1)$, if one takes
into account the periodicity of $y$. For the choice \eqref{Raff:ex1}
of $R_\mathrm{aff}$ the first condition fixes the phase of the
solution $y$ (serving as a phase condition) and the second condition
fixes the amplitude, locally parametrizing the solution family by the
constant $y_0$ (thus, serving as a pseudo-arclength condition). System
\eqref{dde:ex1:rescaled}--\eqref{Raff:ex1} will have solutions of the
form $y(t)=y_0\sin(2\pi t)+O(|y_0|^2)$, $T=2\pi+O(|y_0|^2)$, $p=\pi/2+O(|y_0|^2)$ for small
$y_0$ according to the Hopf bifurcation theorem \citep{S12}. The
unknowns in this problem are $(y(\cdot),T,p)$.

The functional $G(y,T,p)$ for this example and its derivative (defined
on $C^1_\pi\times\R^2$) have already been used for illustration in
\eqref{intro:dde:ex:rescaled} and in \eqref{intro:DG:ex}.

\section{Fixed-point problem equivalent to periodic BVP and relevant notation}
\label{sec:formulation}
In the formulation \eqref{res:bvp} the unknown period $T$ plays the same
role as a problem parameter. Thus, in the following we collect $T$ and $p$ into a $(n_p+1)$-dimensional parameter vector $\mu=(T,p)$, introducing the
nonlinear functional 
\begin{align}
  G(y,\mu)&=T\gfde(y((\cdot)/T),p)\mbox{,\quad where now $\mu=(T,p)\in\R^{n_\mu}$, $n_\mu=n_p+1$,}\label{gen:Gmu}
\end{align}
and abbreviate $R_\mathrm{aff}[y,\mu]=R_\mathrm{aff}[y,T,p]$
to simplify notation. Thus \eqref{res:bvp} is a special case of the general BVP
\begin{align}\label{bvp1mu}
  \begin{aligned}[t]
    y'(t)&=G(y_{t},\mu),&&\in\R^{n_y}\mbox{\ for $t\in[0,1)$,}&&\mbox{(FDE for $1$-periodic continuous $y$),}\\
    0&=R_\mathrm{aff}[y,\mu] &&\in\R^{n_\mu}&&\mbox{(affine constraints).}
  \end{aligned}
\end{align}

\noindent We reformulate \eqref{bvp1mu} as fixed-point problem for the solution, obtained from the differential
equation through the variation-of-constants formula, following
\cite{andoSIAM2020},
\begin{align}
  \label{fp:nonperiodic}
  y(t)&=y(0)+\int_0^tG(y_s,\mu)\dd s.
\end{align}
We plan to pose the fixed-point problem in a space of periodic
functions of period $1$. However, the right-hand side of equation~\eqref{fp:nonperiodic}
is not guaranteed to be $1$-periodic even if $y$ is $1$-periodic. To
enforce period $1$, we subtract the average of the integrand and then
impose that this average is zero as a separate equation, replacing the
periodic boundary condition. We also introduce the new variable
$\alpha\in\R^{n_y}$, which will be equal $y(0)$ in the
solution. Hence, BVP \eqref{res:bvp} is equivalent to the  fixed-point problem
\begin{align}
  \label{gen:fixedpoint}
  x=\Phi(x),
\end{align}
%
for the operator
\begin{equation}\label{Phi1}
\Phi_{}(x):=
\begin{bmatrix*}[l]
t\mapsto&\alpha+\displaystyle\int_0^tG(v_s,\mu)\dd s-t\int_0^1G(v_s,\mu)\dd s\\[1ex]
&\alpha + \displaystyle\int_0^1G(v_s,\mu)\dd s\\[1ex]
&\mu+R_\mathrm{aff}[v,\mu]
\end{bmatrix*}\mbox{\quad for }x=\begin{bmatrix}
    v(\cdot)\\\alpha\\\mu
\end{bmatrix}\mbox{,}
\end{equation}
where $v$ is $1$-periodic with $v(t)\in\R^{n_y}$, $\alpha\in\R^{n_y}$,
$\mu\in\R^{n_\mu}$. 
Fixed-point problem \eqref{gen:fixedpoint}
is then equivalent to the original problem of finding a periodic solution of FDE~\eqref{intro:fde} in the sense that for every periodic solution $y(t)$ with period $T>0$ of \eqref{intro:fde} at parameter $p$, there exists a phase shift $\theta\in\R$ such that $x=(t\mapsto y(\theta+tT),y(\theta),T,p)$ is a solution of \eqref{gen:fixedpoint}, and, vice versa, for a fixed point $x=(v,\alpha,\mu)=(v,\alpha,T,p)$ of $\Phi$ the function $y(t)=v(t/T)$ is a periodic solution of \eqref{intro:fde} with period $T$ at parameter $p$, $\alpha$ must be equal to $v(0)$, and $(v,T,p)$ satisfy the constraints $R_\mathrm{aff}[v,T,p]$.

\paragraph{Notation for function spaces and norms}
The definition of $\Phi$ in \eqref{Phi1} had not specified the function space for the $v$ component of $x$ yet as we have not discussed what type of perturbation a discretization may introduce  for $\Phi$.
To specify suitable spaces we use the notation
\begin{align*}
  C^{k\phantom{,1}}_\pi&\!=\{v:\mbox{\ $k\times$\,cont.\,diff.,\,} v(t)=v(t+1)\mbox{\ for all $t$}\},& 
\|v\|_{k\phantom{,1}}&\!=\max_{t\in[0,1],j\leq k}|v^{(j)}(t)|,\\
  C^{k,1}_\pi&=\{v\in C^k_\pi,\lip v^{(k)}<\infty\},& 
\|y\|_{k,1}&\!=\max\{\|v\|_k,\lip v^{(k)}\},\\
L^\infty_\pi&=\{v\mbox{\ ess.bd.,\ } v(t)=v(t+1)\mbox{\ for all $t$}\},& 
\|v\|_\infty\ &=\operatorname{ess\,sup}_{t\in[0,1]}|v(t)|.
\end{align*}
The dimension of the function's value $v(t)$ is determined by
context such that we often drop domain and codomain indicators in the
spaces. Otherwise we write, e.g., $C^{k,j}_\pi(\R^{n_u})$ or $L^\infty_\pi(\R^{n_u})$.
\paragraph{Linear and nonlinear part of fixed-point problem $x=\Phi(x)$}
The variable $x$ has several components, $x=(v,\alpha,\mu)$,
introduced above, of which only the first one, $v$, is
infinite-dimensional such that all norms of spaces for $v$ can be trivially
extended by the finite-dimensional maximum norms of $\alpha$ and $\mu$. Hence, 
we define the extended
spaces
\begin{align*}
  C^{k,j}_\mathrm{e}&=C^{k,j}_\pi\times\R^{n_y}\times\R^{n_\mu},&
  L^\infty_\mathrm{e}&=L^\infty_\pi\times\R^{n_y}\times\R^{n_\mu},
\end{align*}
and continue to use the $\|\cdot\|_{k,j}$ or $\|_\cdot\|_\infty$
notation for their respective norms. We split the operator $\Phi$,
defined in \eqref{Phi1}, into a compact linear part $\mathcal{L}$ and
a nonlinear part $g$, such that $\Phi=\mathcal{L}\circ g$. For
$\mathcal{L}$ and $g$ we can now pick specific spaces:
\begin{align}\label{g_ext}
  g&:C^{0,1}_\mathrm{e}\to C^0_\mathrm{e}\mbox{,\quad with}&
  g
  \left(\begin{bmatrix}
      v\\\alpha\\\mu
    \end{bmatrix}\right)
  &=
  \begin{bmatrix}
    t\mapsto G(v_t,\mu)\\
    \alpha\\
    \mu+R_\mathrm{aff}[v,\mu]
  \end{bmatrix}\mbox{,\quad and}\\
  \mathcal{L}&:
  \begin{aligned}[t]
    &L^{\infty}_\mathrm{e}\to C^{0,1}_\mathrm{e}\\
    \makebox[0pt]{or\qquad}&C^\ell_\mathrm{e}\to C^{\ell+1}_\mathrm{e}
  \end{aligned}
  \mbox{,\quad with}&
  \label{integral_ext}
  \mathcal{L}
  \begin{bmatrix}
    w\\\alpha\\\nu
  \end{bmatrix}\hspace*{0.8em}
  &=\begin{bmatrix}
    \displaystyle t\mapsto \alpha+\int_0^tw(s)\dd s-t\int_0^1w(s)\dd s \\[2mm]
    \alpha+\displaystyle\int_0^1w(s)\dd s\\
    \nu
  \end{bmatrix}.
\end{align}
Convergence Theorem~\ref{thm:conv} for fixed-point problem \eqref{gen:fixedpoint} proves that the fixed
point $x^L$ of a discretized operator $\Phi_L$ converges with the rate
expected by the order of the discretization scheme to a fixed point
$x^*$ of $\Phi$ under appropriate conditions. This convergence will
operate on a small ball $B^{0,1}_r(x^*)$ of Lipschitz continuous
functions in $C^{0,1}_\mathrm{e}$ around $x^*$. The center $x^*$ has higher-order
regularity than $C^{0,1}_\mathrm{e}$: $x^*\in C^{\lm+1}_\mathrm{e}$, where $\lm\geq0$ depends on
regularity assumptions on the right-hand side $G$.

\paragraph{Discretized fixed-point problem}
For functions $z\in C^0_\pi$ we define the interpolation
projection $P_Lz$ as the unique piecewise
polynomial on mesh $(t_i)_{i=0}^L$ of degree $m-1$ where the piece on each interval $(t_{i-1},t_i)$ equals $z$ on the nodes $t_{i,j}$:
\begin{align*}
  P_L:&\phantom{=\ }C^0_\pi\mapsto L^\infty_\pi,\\
  P_Lz&=\hat{z}\quad \begin{aligned}[t]
    &\mbox{with $\hat{z}(t_{i,j})=z(t_{i,j})$ for all $i\in\{1,\ldots,L\},j\in\{1,\ldots,m\}$,}\\
    &\mbox{and\ $\hat{z}$ is degree $m-1$ polynomial on $(t_{i-1},t_i)$ for all $i\in\{1,\ldots,L\}$.}
\end{aligned}
\end{align*}
In the name $P_L$ we do not indicate the dependence on the interpolation degree $m-1$ as we will keep this degree constant, studying only the limit $L\to \infty$ in our convergence analysis. The dependence on the mesh $(t_i)$, which will change with increasing $L$, is also implicitly included in the subscript $L$.
The interpolating piecewise polynomial $t\mapsto [P_Lz](t)$ is not necessarily
continuous as it may have discontinuities at the mesh boundaries
$t_i$, such that the codomain of $P_L$ is $L^\infty_\pi$. By construction of $\Phi=\mathcal{L}\circ g$ and $P_L$ we have the following equivalence.
\begin{lemma}[Equivalence of discretized fixed point problem]\label{thm:disc:fp}
  The discretized BVP \eqref{res:bvp:disc} is equivalent to the fixed point problem
  \begin{align}
    \label{thm:disc:fp:eq}
    x&=\mathcal{L}\mathcal{P}_Lg(x)=:\Phi_L(x),
  \end{align}
  where $\mathcal{L}:L^\infty_\mathrm{e}\to C^{0,1}_\mathrm{e}$ and
  $g:C^{0,1}_\mathrm{e}\to C^0_\mathrm{e}$ are defined in \eqref{g_ext}
  and \eqref{integral_ext}, and $\mathcal{P}_L$ is the trivial extension of $P_L$,
    \begin{align}\label{def:interp:PL}
    \mathcal{P}_L:C^0_\mathrm{e}\to
       L^\infty_\mathrm{e},\qquad
    \mathcal{P}_L
    \begin{bmatrix}
      w\\\alpha\\\nu
    \end{bmatrix}
    =
       \begin{bmatrix}
         P_Lw\\
         \alpha\\
         \nu
       \end{bmatrix}.
  \end{align}
\end{lemma}
\begin{remark}[Discretized solution space]\label{rem:solspace}
  In our notation the discretized fixed-point operator $\Phi_L=\mathcal{L}\circ \mathcal{P}_L\circ g$ has the codomain 
  \begin{align*}
    \rg\mathcal{L}\mathcal{P}_L=\{\mbox{$p\in C^{0,1}_\pi$, $p\vert_{[t_{i-1},t_i]}$ $m$-degree poly.\ for $i=1,\ldots,L$}\}\times\R^{n_y}\times\R^{n_\mu}\subset C^{0,1}_\mathrm{e}.
  \end{align*}
  In particular, if $x=(v,\alpha,\mu)\in\rg\mathcal{L}\mathcal{P}_L$,
  then its first component $v\in C^{0,1}_\pi$ is Lipschitz continuous, but cannot be expected to be
  continuously differentiable.
\end{remark}
\paragraph{Example} We illustrate Remark~\ref{rem:solspace} for our rescaled example  \eqref{intro:dde:ex:rescaled}.  The first component of $Dg(x)\delta^x$ in $x=(y,T,p)$ and $\delta^x=(\delta^y,\delta^T,\delta^p)$ is
a function of time that has the form
$D\gts(y,T,p)[\delta^y,\delta^T,\sigma^p]$ and is given in
\eqref{intro:DG:ex}. It contains the term
$y'(t-p/T-y(t)/T)$, which is not defined for all $t$ if
$y\in C^{0,1}_\pi$. In particular, when attempting to evaluate
$D\gts(y,T,p)$ for a function $y$ in the discrete solution space
$\rg\mathcal{L}\mathcal{P}_L$, one encounters an ill-defined term
whenever a mesh point $t_k$ and a collocation point $t_{i,j}$ satisfy
\begin{align*}
  t_{i,j}-p/T-y(t_{i,j})/T-t_k\in\Z\mbox{\quad for some $i,k\in\{1,\ldots,L\}$, $j\in\{1,\ldots,m\}$,}
\end{align*}
because in the mesh boundary points $t_k$ the right-sided derivative
and the left-sided derivative are generally different (see
\eqref{res:bvp:disc} for introduction of collocation points $t_{i,j}$
and mesh points $t_k$).

\section{Mildly differentiable nonlinear functionals}
\label{sec:mild}

\paragraph{Consequences of mild differentiability for nonlinearities of DDEs}
Let $G:C^0_\pi\mapsto\R^{n_G}$ be mildly differentiable to order $\lm\geq1$. We define $  \gts:C^0_\pi\to C^0_\pi$ as
\begin{align}
  \label{def:gts}
  \gts(u)(t)=G(u_t)
\end{align}
(domain and codomain have possibly different dimensions $n_u$ and
$n_G$), which combines the nonlinear functional $u\mapsto G(u)$ with
the time shift $u\mapsto u_t$.  The map $g$ in our fixed-point problem
\eqref{Phi1} is of this type in its first component,
$(v,\alpha,\mu)\mapsto [t\mapsto G(v_t,\mu)]$.  The continuity of
$\gts(u)$ in time and the continuity of $\gts$ in $u$ in the
$C^0$-norm follow from the continuity of $(t,u)\mapsto u_t$ and
$G:C^0_\pi\to\R^{n_G}$.  The following Lemmas collect relevant
conclusions for this type of nonlinearity, which will be needed to
obtain results on the regularity of a fixed point $x^*$ of
$\Phi$ as well as the convergence of the fixed point $x^L$ of the
discrete version $\Phi_L$ to $x^*$.
\begin{lemma}[Extended differentiability of nonlinearity with time shift]\label{thm:diff:shift}\label{thm:p:g:diff}
  Assume that $G:C^0_\pi\to\R^{n_G}$ satisfies mild differentiability
  to at least order $\lm$ according to Definition~\ref{def:mild} and
  define $\gts(u)(t)=G(u_t)$ for $u\in C^0_\pi$. Then, for
  $\ell\leq\lm$ the map $\gts:C^{\lm}_\pi\to C^{\lm-\ell}_\pi$ is
  $\ell$ times continuously differentiable \textup{(}for $\ell=0$
  continuous\textup{)}. The map
    \begin{align*}
      C^\lm_\pi\times(C^{\lm-1}_\pi)^\ell\ni (u,\delta^{u,1},\ldots,\delta^{u,\ell})\mapsto D^\ell \gts(u)\delta^{u,1}\ldots \delta^{u,\ell}\in C^{\lm-\ell}_\pi
    \end{align*}
    is continuous.
  \end{lemma}
  See Appendix~\ref{sec:proof:diff:shift} 
  for a detailed proof for this lemma. The case $\lm=1$ for
  Lemma~\ref{thm:diff:shift} illustrates where the
  extendability condition for $G$ in the definition of mild
  differentiability is needed for differentiability of $\gts(u)$ in
  time.  Formally, one ``applies the chain rule'', differentiating
  $G(u_t)$ with respect to $t$ with $u\in C^1_\pi$ (for $k=1$), such
  that $(\gts(u))'=[G(u_t)]'=D G(u_t)u_t'$, but since
  $(u_t)'\in C^0_\pi$ the expression $D G(u_t)u_t'$ is only valid and
  continuous because of the continuous extendability of
  $DG(u_t)$ to $C^0_\pi$ (in which $(u_t)'$ lies). This is
  point~\ref{def:mild:ext} in Definition~\ref{def:mild}. The general
  case ($\lm>1$) then applies this argument repeatedly. In particular,
  we see that $\gts$ is only continuously differentiable $\ell$ times
  from a higher-regularity space $C^\lm_\pi$ to the lower-regularity space $C^{\lm-\ell}_\pi$.

We observe that the finite difference limit can be extended to Lipschitz continuous deviations (formulated and needed only for $\ell=1$) by continuity:
\begin{corollary}[Extension of finite-difference limit]\label{thm:findiff}
  Assume that $G:C^0_\pi\to\R^{n_G}$ is mildly differentiable once. Then the map
  $C^1_\pi\times C^0_\pi\ni (u,\delta^u)\mapsto D\gts(u)\delta^u\in
  C^0_\pi$ defined as in \eqref{def:gts} satisfies for all $u\in C^1_\pi$
  \begin{equation}\label{thm:gcontdiff}
    \lim_{
      \begin{subarray}{c}
        v,w\in C^{0,1}_\pi\\[0.2ex]
        \|v\|_{0,1},\|w\|_{0,1}\to0
      \end{subarray}
    }\cfrac{\|\gts(u+v)-\gts(u+w)-
      D\gts(u)[v-w]\|_0}{\|v-w\|_{0,1}}=0\mbox{.}
  \end{equation}
\end{corollary}
\noindent
Note that in the finite-difference quotient limit
\eqref{thm:gcontdiff}, we extend from the $C^1$-norm to the
$C^{0,1}$-norm. This is possible because the derivative $DG$ of $G$
(and, hence, the derivative $D\gts$ of $\gts$) is continuously
extendable to $C^0_\pi$ in its linear argument.

\

The regularity properties outlined in this section imply, in particular, that we cannot  rely on the general framework by \cite{mas15NM} to prove the convergence of the numerical method, in contrast to \citep{andoSIAM2020}, which studied the constant delay case. In other words, it is not possible to prove all the (theoretical and numerical) assumptions  made in \citep{mas15NM} to reach our convergence result. In particular, some assumptions 
will only hold in a weaker form, with different norms in the relevant inequalities than those in the general framework in \citep{mas15NM}. In section \ref{sec:convergence} we will show that this does not impede high-order convergence for polynomial collocation methods, such that we obtain convergence results identical to those for constant delays.

\section{Assumptions on the (infinite-dimensional) problem and immediate consequences}
\label{sec:assumptions}
This section reformulates Assumption~\ref{res:ass}, stated in
Section~\ref{sec:result} for the BVP, as assumptions on the
infinite-dimensional fixed-point problem \eqref{Phi1} for $\Phi$. These are fewer than stated for the general theory by \cite{mas15NM}, because Maset includes further assumptions needed for determining the convergence rate of the Newton iterations when
solving the discretized problem.
\subsection*{Existence and regularity of solution to infinite-dimensional problem} 
\begin{assumption}[Existence of solution]\label{ass:existence}The fixed-point problem \eqref{Phi1} has a solution
$x^*$\textup{:} $x^*=\Phi(x^*)$.
\end{assumption}
We denote the components of $x^*$ as $(v^*,\alpha^*,\mu^*)$. By construction of the fixed point problem, $v^*$ will
be Lipschitz continuous and periodic with period $1$, and
satisfy the differential equation
\begin{align}
(v^*)'(t)&=G(v^*_{t},\mu^*),\label{G:simple:DEv}
\end{align}
such that $v^*$ will even be in $C^1_\pi$. 
The Hopf bifurcation theorem, proved for functional differential equations with state-dependent equations in \cite{S12}, provides a scenario that ensures the existence of periodic orbits, and, hence, solutions of fixed-point problem~\eqref{Phi1}. The illustrative example~\eqref{intro:dde:ex} is chosen to satisfy the assumptions of the Hopf bifurcation theorem for $p\approx \pi/2$ and small-amplitude periodic solutions $y(t)$ with period $2\pi$ (hence, solutions $(y(\cdot),T,p)$ of \eqref{dde:ex1:rescaled}, \eqref{Raff:ex1} with small-amplitude $y$).
\begin{assumption}[Mild differentiability of right-hand side $G$]\label{ass:mdiff}
  The right-hand side $G$ of the FDE \eqref{bvp1mu} is mildly
  differentiable to order $\lm\geq 1$.
\end{assumption}
As we include the parameters (treating constants as special cases of
$1$-periodic functions), the dimensions are $n_u=n_y+n_\mu=n_y+n_p+1$
for the argument of $G$ and $n_G=n_y$ for the value of $G$.

Lemma~\ref{thm:diff:shift} implies the following corollary about the
regularity of the solution $x^*$ of the fixed-point problem
\eqref{Phi1}, $x=\Phi(x)$.
\begin{corollary}[Regularity of solution of fixed-point problem \eqref{Phi1}]\label{thm:regularity:fp}
  Let $x^*=(v^*,\alpha^*,T^*)$ be a solution of the fixed-point
  problem \eqref{Phi1}, and let the right-hand side $G$ of the FDE
  \eqref{bvp1mu} satisfy mild differentiability to order $\lm$. Then
  the solution component $v^*$ is in $C^{\lm+1}_\pi$
  as a function of time, and, hence, $x^*\in C^{\lm+1}_\mathrm{e}$.
\end{corollary}
The statement of Corollary~\ref{thm:regularity:fp} follows from
Lemma~\ref{thm:diff:shift},  applied to the
case $\ell=0$, and using that $v^*$ satisfies the differential equation
\eqref{G:simple:DEv}: for each $j$ from $0$ to $\lm$ we have that
$G(v^*_t,\mu^*)$ is in $C^j_\pi$ because $v^*$ is in $C^j_\pi$ by
Lemma~\ref{thm:diff:shift}. Then by the differential equation
\eqref{G:simple:DEv} $(v^*)'(t)$ is in $C^j_\pi$, such that $v^*$ is
in $C^{j+1}_\pi$.

Hence, by Assumption~\ref{ass:mdiff} that $\lm\geq1$, we have that
$x^*\in C^{\lm+1}_\mathrm{e}\subseteq C^2_\mathrm{e}$.
Note that Assumption
\ref{ass:mdiff} of mild differentiability of $G$ with $\lm=1$,
together with Corollary \ref{thm:regularity:fp}, imply that the derivative
\begin{align*}
  C^1_\pi\ni x\mapsto D\Phi_{}(x)=\mathcal{L}Dg(x)\in\Lin(C^1_\pi;C^1_\pi)
\end{align*}
 of the fixed point map $\Phi$ defined in \eqref{Phi1} exists and is
continuous in $x=x^*$. This follows from the continuous
differentiability of $g$ as a map from $C^1_\mathrm{e}$ to
$C^0_\mathrm{e}$ and the subsequent application of the linear
continuous mapping
$\mathcal{L}\in \operatorname{Lin}(C^0_\mathrm{e}; C^1_\mathrm{e})$,
which increases regularity by one degree.

The last assumption that we make on the infinite-dimensional problem is the well-posedness of the system linearized around the fixed point, which will be needed to show stability of the discretized problem.
\begin{assumption}[Well-posedness of infinite-dimensional linear problem]\label{ass:linear:inv}
 The linear bounded operator $I-D\Phi_{}(x^{\ast})$ is injective on $C^1_\mathrm{e}$, that is, if $y=D\Phi(x^*)y$ and $y\in C^1_\mathrm{e}$ then $y=0$.
\end{assumption}
We denote the norm of the inverse
\begin{align}
  \label{def:Cstab:infty}
  C_{\mathrm{stab},\infty}=\|[I-D\Phi(x^*)]^{-1}\|_{C^1_\mathrm{e}\leftarrow C^1_\mathrm{e}}=\|[I-D\Phi(x^*)]^{-1}\|_{C^{0,1}_\mathrm{e}\leftarrow C^{0,1}_\mathrm{e}}.
\end{align}
Elements $y$ of the nullspace of $[I-D\Phi(x^*)]$ that are at least in
$C^0_\mathrm{e}$ satisfy the identity $y=\mathcal{L}Dg(x^*)y$, such
that they are in the image of $\mathcal{L}$ of $C^0_\mathrm{e}$, which
is in $C^1_\mathrm{e}$, since $\mathcal{L}$ involves an
anti-derivative. Hence, we may replace $C^1_\mathrm{e}$ in
Assumption~\ref{ass:linear:inv} by $C^0_\mathrm{e}$ (and, even
$L^\infty_\mathrm{e}$).

Since $D\Phi(x^*)$ is a compact linear operator on spaces
$C^{k,j}_\mathrm{e}$ for all $k\geq0$ and $j=0,1$, the nullspace of
$D\Phi(x^*)$ is at most finite-dimensional, and
$\dim\ker [I-D\Phi(x^*)]=0$ implies the existence of a bounded inverse
$[I-D\Phi(x^*)]^{-1}$. 
A sequence of results formulate Assumption~\ref{ass:linear:inv} in
terms of requiring full rank of a finite-dimensional \emph{characteristic} matrix. One example construction for a characteristic matrix of a linear time-periodic FDE is given by \cite{SS11}, 
generalized to FDEs of neutral type by \cite{VerduynLunel2023}. In practice numerical methods for determining eigenvalues of monodromy operators, such as the pseudospectral method \citep{blv22,breifac06,bmvsinum12}, work directly on large matrices generated by the discretization projection $\mathcal{P}_L$, without prior reduction. \cite{borgioli2020pseudospectral} proves convergence of methods as implemented in \textsc{DDE-Biftool} and \textsc{knut}.
\section{Convergence of solutions of discretized problem}
\label{sec:convergence}
In this section, we show that solutions $x$ of the fixed point problem
for the discretized map given in the Equivalence Lemma
\ref{thm:disc:fp},
\begin{align}
  \Phi_L=\mathcal{L}\mathcal{P}_Lg,\label{def:PhiL}
\end{align}
are locally unique, and converge to the fixed point $x^*$ of $\Phi$
(solving $x^*=\Phi(x^*)=\mathcal{L}g(x^*)$).  In $x^* $ (which is in
$C^2_\mathrm{e}\subset C^1_\mathrm{e}$) the derivative of $\Phi_L$ is
well defined and equal to
$D\Phi_L(x^*)=\mathcal{L}\mathcal{P}_LDg(x^*)$. The discretized fixed
point problem $\Phi_L(x)=x$ can be formulated in terms of
$\delta^x:=x-x^*$ as
\begin{align}\label{dfp2:split}
  \lefteqn{[I-D\Phi_L(x^*)]\delta^x=}\\
  =&\phantom{+}\mathcal{L}(\mathcal{P}_L-I)g(x^*)&&\mbox{(consistency term $\epsilon_\mathrm{c}(L)$)}\label{dfp2:consistency}\\
   &+\mathcal{L}\mathcal{P}_L[g(x^*+\delta^x)-g(x^*)-Dg(x^*)\delta^x]&&\mbox{(nonlinearity term $\epsilon_\mathrm{nl}(L,\delta^x)$).}\label{dfp2:nonlinearity}     
\end{align}
Given the left-hand side of \eqref{dfp2:split}, a necessary condition
for the sought well-posedness is the invertibility of the operator
$I-D\Phi_L(x^*)$ uniformly for $L\to\infty$, i.e., the stability of
the discretized problem. This is particularly evident in the case of a
linear $\Phi$, where $\epsilon_\mathrm{nl}(L,\delta^x)=0$ and
$\delta^x$ only appears in the left-hand side. Since
\begin{align*}
  [I-D\Phi_L(x^*)]\delta^x=&[I-D\Phi(x^*)]\delta^x-[D\Phi_L(x^*)-D\Phi(x^*)]\delta^x,
\end{align*}
the stability of the discretized problem is
determined by the invertibility of the original infinite-dimensional
problem --- guaranteed by Assumption \ref{ass:linear:inv} --- provided
that the linearization of the discretization in $x^*$ approximates the
linearization of the infinite-dimensional problem arbitrarily
well. Indeed, this approximation holds for large $L$, as the following
lemma states.
\begin{lemma}[Consistency of derivative of $\Phi_L$]
\label{thm:DPhi:consistency:alt}
Let $x^*=(v^*,\alpha^*,\mu^*)$ be a fixed point of $\Phi$, let
$G:C^0_\pi\to\R^{n_y}$ be mildly differentiable once.  Then there exists a monotone increasing continuous function $\omega^*:[0,\infty)\to[0,\infty)$ with $\omega^*(0)=0$, such that
\begin{align}
  \label{thm:DPhi:consistency:ineq:alt}
  \|D\Phi_L(x^*)-D\Phi(x^*)\|_{C^{0,1}_\mathrm{e}\leftarrow C^{0,1}_\mathrm{e}}&=\|\mathcal{L}[I-\mathcal{P}_L]Dg(x^*)\|_{C^{0,1}_\mathrm{e}\leftarrow C^{0,1}_\mathrm{e}}\leq \omega^*(1/L).
\end{align}
\end{lemma}
\begin{proof}
  The closed set $B^{0,1}_1(0)$ of functions with Lipschitz norm less
  than or equal to unity is compact in $C^0_\mathrm{e}$. By mild
  differentiability of $g$ the map $Dg(x^*)$ is in
  $\Lin(C^0_\mathrm{e};C^0_\mathrm{e})$, so a continuous (bounded)
  linear map from $C^0_\mathrm{e}$ into itself.
  Consequently, the set
  $\mathcal{S}:=\{Dg(x^*)y:\|y\|_{0,1}\leq 1\}$ is compact in
  $C^0_\mathrm{e}$, and, hence, uniformly equicontinuous. This means
  that there exists a uniform modulus of continuity for this set
  $\mathcal{S}$, a continuous monotone increasing function
  $\omega_g:[0,\infty)\to[0,\infty)$ with $\omega_g(0)=0$, such that
  \begin{align*}
    |[Dg(x^*)y](t+h)-[Dg(x^*)y](t)|\leq \omega_g(h)\mbox{\quad for all $t\in\R$ and $y$ with $\|y\|_{0,1}\leq1$.}
  \end{align*}
  Furthermore, the interpolation approximation $\mathcal{P}_L$ satisfies for any continuous function $z$ (see \cite{riv69})
  \begin{align}
    \|z-\mathcal{P}_Lz\|_\infty \leq 6(1+\Lambda_m)\omega_z(C_\mathrm{msh}/(2mL))\label{PL:genconv:mod}
  \end{align}
  where $\Lambda_m$ is the
  Lebesgue constant  for interpolation at the points
  $(t_{\mathrm{c,}j})_{j=1}^m$ on the interval $[0,1]$ chosen in the
  discretization \eqref{res:bvp:disc}, and $\delta\mapsto \omega_z(\delta)$ is the modulus of continuity for $z$. Thus, the equicontinuity with modulus of continuity $\omega_g$ on $\mathcal{S}$ implies that  
  \begin{align}
    \|[I-\mathcal{P}_L]Dg(x^*)y\|_\infty \leq 6(1+\Lambda_m)\omega_g(C_\mathrm{msh}/(2mL))
    \mbox{\quad if $\|y\|_{0,1}\leq 1$.}\label{PL:modcont}
  \end{align}
  Defining
    $\omega^*(s)=6(1+\Lambda_m)\|\mathcal{L}\|_{C^{0,1}_\mathrm{e}\leftarrow L^\infty_\mathrm{e}}\omega_g(C_\mathrm{msh}\,s/(2m))$,
  \eqref{PL:modcont} implies
the claim of the lemma.
\end{proof}
\begin{remark}[Sharper estimate if $G$ is mildly differentiable twice]
  If $G$ is mildly differentiable twice, then $Dg(x^*)\in\Lin(C^1_\mathrm{e};C^1_\mathrm{e})$, such that the sharper estimate
  \begin{align}
    \label{rem:DPhi:consistency:sharp}
  \|D\Phi_L(x^*)-D\Phi(x^*)\|_{C^{0,1}_\mathrm{e}\leftarrow C^{0,1}_\mathrm{e}}\leq \frac{3(1+\Lambda_m)}{m}\frac{\|\mathcal{L}\|_{C^{0,1}_\mathrm{e}\leftarrow L^\infty_\mathrm{e}}\|Dg(x^*)\|_{C^1_\mathrm{e}\leftarrow C^1_\mathrm{e}}}{L}    
  \end{align}
  holds (replacing $\omega_g$ with $\|Dg(x^*)\|_{C^1_\mathrm{e}\leftarrow C^1_\mathrm{e}}\|$).
\end{remark}
The stability of the discretized problem is a straightforward consequence of the previous lemma.
\begin{corollary}\label{thm:stabilityPhiL}
  Under
  Assumptions~\ref{ass:existence},\,\ref{ass:mdiff},\,\ref{ass:linear:inv} from
   section~\ref{sec:assumptions} the fixed point problem
  $\Phi_L(x)=x$ is stable in $x^*$, that is, there exists a bound
  $C_{\mathrm{stab},L}$ for
  $\|[I-D\Phi_L(x^*)]^{-1}\|_{C^{0,1}_\mathrm{e}\leftarrow
    C^{0,1}_\mathrm{e}}$ for sufficiently large $L$, and
  $C_{\mathrm{stab},L}\to C_{\mathrm{stab},\infty}$ for $L\to\infty$.
\end{corollary}
\begin{proof}
By Assumption \ref{ass:linear:inv}, $I-D\Phi(x^*)$ has a bounded inverse. The Banach perturbation lemma (e.g., \cite[Theorem 2.1.1]{ort90}) and Lemma~\ref{thm:DPhi:consistency:alt} then ensure that $I-D\Phi_L(x^*)=I-\mathcal{L}\mathcal{P}_LDg(x^*)$ has a bounded
inverse, too, for sufficiently large $L$. Let $L_\mathrm{diff}$ be such that
  $\omega^*(1/L_\mathrm{diff})\leq 1/(2C_{\mathrm{stab},\infty})$.
Then, for $L\geq L_\mathrm{diff}$, the inverse has the norm
\begin{align}\nonumber
  \|[I-D\Phi_L(x^*)]^{-1}\|_{C^{0,1}_\mathrm{e}\leftarrow C^{0,1}_\mathrm{e}}&\leq
  \cfrac{C_{\mathrm{stab},\infty}}{1-\omega^*(1/L)C_{\mathrm{stab},\infty}}\\
    \label{def:CstabL}
    &\leq (1+2C_{\mathrm{stab},\infty}\omega^*(1/L))C_{\mathrm{stab},\infty}=:C_{\mathrm{stab},L}.
\end{align}
\end{proof}
\noindent
A rough upper bound for the norm of the inverse of $I-D\Phi_L(x^*)$ for $L\geq L_\mathrm{diff}$ is, thus,
\begin{align}
  \label{def:Cstab}
  C_\mathrm{stab}:=C_{\mathrm{stab},L_\mathrm{diff}}=(1+2C_{\mathrm{stab},\infty}\omega^*(1/L_\mathrm{diff}))C_{\mathrm{stab},\infty}\leq 2C_{\mathrm{stab},\infty}.
\end{align}
Using Corollary \ref{thm:stabilityPhiL}, we can isolate $\delta^x$
in the identity \eqref{dfp2:split}:
\begin{equation}\label{rhs:contraction}
\delta^x=[I-\Phi_L(x^*)]^{-1}[\epsilon_\mathrm{c}(L)+\epsilon_\mathrm{nl}(L,\delta^x)].
\end{equation}
In order to prove that the right-hand side of \eqref{rhs:contraction} is a contraction --- and, thus, defines a well-posed fixed-point problem --- we need suitable bounds for the consistency error $\epsilon_\mathrm{c}(L)$ and the nonlinear part $\epsilon_\mathrm{nl}(L,\delta^x)$. For the consistency error the smoothness of the solution $x^*$, established in Corollary~\ref{thm:regularity:fp}, enables us to apply the convergence theory for interpolation of smooth functions in the following lemma.
\begin{lemma}\label{lemma:cons}
  Let $x^*$ be a fixed point of $\Phi$ and $G$ be mildly differentiable to order $\lm\geq1$. Then, the consistency error $\epsilon_\mathrm{c}(L)=\mathcal{L}(\mathcal{P}_L-I)g(x^*)$ in \eqref{dfp2:consistency} satisfies
\begin{align*}
\|\epsilon_\mathrm{c}(L)\|_{0,1}=O(L^{-\min\{\lm,m\}}),
\end{align*}
where $m-1$ is the degree of the interpolation polynomial used in $\mathcal{P}_L$ in \eqref{def:interp:PL}.
\end{lemma}
\begin{proof}
By Corollary~\ref{thm:regularity:fp}, $x^*$ is in $C^{\lm+1}_\mathrm{e}$ and, thus, $g(x^*)$ is in  $C^\lm_\mathrm{e}$. Let $[t_i,t_{i+1}]$ be one of the subintervals in the mesh of the discretization, where $i\in\{0,\ldots,L-1\}$. If $\lm\geq m$, then
\begin{align}
\|(\mathcal{P}_L-I)g(x^*)\vert_{[t_i,t_{i+1}]}\|_{\infty}\leq \frac{\|[g(x^*)]^{(m)}\vert_{[t_i,t_{i+1}]}\|_0}{m!}[t_{i+1}-t_i]^{m}\label{disc:error:mleqlmax}  
\end{align}
follows by the Cauchy interpolation remainder theorem (see, e.g., \cite[Section 6.1, Theorem 2]{kc02}).
If $\lm\leq m$, by \cite[Theorems 1.8-1.9]{arn01},
$$
\|(\mathcal{P}_L-I)g(x^*)\vert_{[t_i,t_{i+1}]}\|_{\infty}\leq c\left[\frac{t_{i+1}-t_i}{2m}\right]^{\lm}\|g(x^*)^{(\lm)}\|_0.
$$
Inserting the upper bound $C_\mathrm{msh}/L$ for the interval lengths, the statement of the lemma follows from the boundedness of the operator $\mathcal{L}:L^{\infty}_\mathrm{e}\to C^{0,1}_\mathrm{e}$.
\end{proof}

\paragraph{Smallness of nonlinear term $\epsilon_\mathrm{nl}(L,\delta^x)$}
The nonlinear term $\epsilon_\mathrm{nl}(L,\delta^x)$ in the identity
\eqref{dfp2:split} is zero if $\delta^x$ is zero. We now want to find
a Lipschitz constant for $\epsilon_\mathrm{nl}$ with respect to
$\delta^x$ that is sufficiently small. The following lemma provides us
with an estimate.
\begin{lemma}\label{lemma:nonlin}
  Let $x^*$ be a fixed point of $\Phi$ and $G$ be mildly differentiable to order $\lm\geq1$. Then, for all $\rho_\mathrm{nl}>0$ there exists $r_\mathrm{nl}>0$ such that all $\delta^{x,1},\delta^{x,2}\in\mathcal{B}^{0,1}_{r_\mathrm{nl}}(0)$ and $L>0$ satisfy
\begin{align*}
\|\epsilon_\mathrm{nl}(L,\delta^{x,1})-\epsilon_\mathrm{nl}(L,\delta^{x,2})\|_{0,1}&\leq\rho_\mathrm{nl}\|\delta^{x,1}-\delta^{x,2}\|_{0,1}\mbox{, implying, in particular,}\\
\|\epsilon_\mathrm{nl}(L,\delta^x)\|_{0,1}&\leq\rho_\mathrm{nl}\|\delta^x\|_{0,1}\mbox{\quad for all $\delta^x\in B^{0,1}_{r_\mathrm{nl}}(0)$.}
\end{align*}
\end{lemma}
\begin{proof} 
  Mild differentiability of $g$ implies by
  Corollary~\ref{thm:findiff} that
  we can find for every $\rho_\mathrm{nl}>0$ a radius
  $r_\mathrm{nl}>0$ such that
   \begin{multline}
     \label{nonlin:estimate0}
     \left\|g(x^*+\delta^{x,1})-g(x^*+\delta^{x,2})-Dg(x^*)\left[\delta^{x,1}-\delta^{x,2}\right]\right\|_0\leq\\
     \frac{\rho_\mathrm{nl}}{\|\mathcal{L}\|_{C^{0,1}_\mathrm{e}\leftarrow L^\infty_\mathrm{e}}\|\mathcal{P}_L\|_{L^\infty_\mathrm{e}\leftarrow C^0_\mathrm{e}}}\left\|\delta^{x,1}-\delta^{x,2}\right\|_{0,1}
   \end{multline}
for all $\delta^{x,1},\delta^{x,2}\in B^{0,1}_{r_\mathrm{nl}}(0)$. In the denominator on the right-hand side the norm $\|\mathcal{P}_L\|_{L^\infty_\mathrm{e}\leftarrow C^0_\mathrm{e}}$ has a uniform upper bound for all $L$ as the degree $m-1$ of the interpolation polynomial in $\mathcal{P}_L$ in \eqref{def:interp:PL} is fixed. Note the stronger $\|\cdot\|_{0,1}$-norm on the right-hand side in \eqref{nonlin:estimate0} required by mild differentiability. Thus, by definition of $\epsilon_\mathrm{nl}$ in \eqref{dfp2:nonlinearity}
\begin{multline*}
\|\epsilon_\mathrm{nl}(L,\delta^{x,1})-\epsilon_\mathrm{nl}(L,\delta^{x,2})\|_{0,1}\leq\\
\|\mathcal{L}\|_{C^{0,1}_\mathrm{e}\leftarrow L^\infty_\mathrm{e}}\|\mathcal{P}_L\|_{L^\infty_\mathrm{e}\leftarrow C^0_\mathrm{e}}
\left\|g(x^*+\delta^{x,1})-g(x^*+\delta^{x,2})-Dg(x^*)\left[\delta^{x,1}-\delta^{x,2}\right]\right\|_0\\
\leq\rho_\mathrm{nl}\|\delta^{x,1}-\delta^{x,2}\|_{0,1}.
\end{multline*}
\end{proof}
\paragraph{Convergence result}
We can now combine Corollary \ref{thm:stabilityPhiL} and Lemmas~\ref{lemma:cons} and \ref{lemma:nonlin} into a convergence theorem.
\begin{theorem}[High-order convergence of collocation]\label{thm:conv}
  Let $G$ be mildly differentiable to order $\lm\geq1$, and let
  $x^*=(v^*,\alpha^*,\mu^*)$ be a fixed point of $\Phi$, as defined in
  \eqref{Phi1}, with bounded inverse of $I-D\Phi(x^*)$, as listed in
  Assumptions~~\ref{ass:existence},\,\ref{ass:mdiff},\,\ref{ass:linear:inv}
  in Section~\ref{sec:assumptions}.

  Then there exists a radius $r>0$ such that the discretized fixed
  point problem $x=\Phi_L(x)$ with $L$ polynomial pieces of degree $m$
  has a unique solution $x^L$ in the ball $B^{0,1}_r(x^*)$ for all
  sufficiently large $L$. The error $\delta^{x,L}=x^L-x^*$ satisfies
$$
\|\delta^{x,L}\|_{0,1}=O(L^{-\min\{\lm,m\}}).
$$
\end{theorem}
\begin{proof}
  After the Lemmas~\ref{thm:DPhi:consistency:alt}, \ref{lemma:cons}
  and \ref{lemma:nonlin} have used mild differentiablity to establish
  estimates for the ingredients of the splitting
  \eqref{dfp2:split}--\eqref{dfp2:nonlinearity}, the arguments in this
  proof are identical to those one would use for smooth right-hand
  sides, explained by \cite{mas15NM} for a general situation. We
  include them here because they show the criteria for how large to
  choose the discretization level $L$.

We first claim that the map $h:C^{0,1}_\mathrm{e}\to C^{0,1}_\mathrm{e}$ given by the right-hand side of \eqref{rhs:contraction}, i.e.,
$$
h(\delta) = [I-D\Phi_L(x^*)]^{-1}[\epsilon_\mathrm{c}(L) + \epsilon_\mathrm{nl}(L,\delta)]
$$
maps back into $B^{0,1}_r(0)$ and is a contraction for a sufficiently
small $r$ (which we need to find) and all sufficiently large $L$. By
the identity \eqref{dfp2:split} fixed points of $h$ are fixed points
of $\Phi_L$.

Let $\kappa\in(0,1)$ be arbitrary and let us recall from the proof of
stability in Lemma~\ref{thm:DPhi:consistency:alt} the upper bound
$C_{\mathrm{stab}}>0$ of
$\|[I-D\Phi_L(x^*)]^{-1}\|_{C^{0,1}_\mathrm{e}\leftarrow
  C^{0,1}_\mathrm{e}}$ given in \eqref{def:Cstab}, which holds for all
$L\geq L_\mathrm{diff}$. We choose a radius
$r=r_\mathrm{nl}$ such that the factor $\rho_\mathrm{nl}$ in
Lemma~\ref{lemma:nonlin} satisfies
$\rho_\mathrm{nl}<\kappa/C_\mathrm{stab}$.
Then we choose the lower bound $L_{\min}$ for $L$,
\begin{align*}
  L_{\min}=&\max\{L_\mathrm{diff},L_\mathrm{c}\}\mbox{,\quad where $L_\mathrm{c}$ is s.\,t.\quad}
 \|\epsilon_\mathrm{c}(L)\|_{0,1}\leq \cfrac{1-\kappa}{C_\mathrm{stab}}\,r_\mathrm{nl}\mbox{\ for all $L\geq L_\mathrm{c}$.}
\end{align*}
For this choice of radius $r$ and $L\geq L_{\min}$ we have for all $\delta^x\in B^{0,1}_r(0)$
\begin{align}
  \nonumber
  \|h(\delta^x)\|_{0,1}&\leq \|[I-D\Phi_L(x^*)]^{-1}\|_{C^{0,1}_\mathrm{e}\leftarrow
  C^{0,1}_\mathrm{e}}\left[\|\epsilon_\mathrm{c}(L)\|_{0,1}+\rho_\mathrm{nl}\|\delta^x\|_{0,1}\right]\\
                       &\leq C_\mathrm{stab}\|\epsilon_\mathrm{c}(L)\|_{0,1}+\kappa \|\delta^x\|_{0,1}\label{convproof:h:bound}\\
  \nonumber
  &\leq (1-\kappa)r_\mathrm{nl}+\kappa \|\delta^x\|_{0,1}\leq r_\mathrm{nl}=r,
\end{align}
such that $h$ maps $B^{0,1}_r(0)$ back into itself.
For the Lipschitz constant of $h$ we have by Lemma \ref{lemma:nonlin}
\begin{align*}
\|\epsilon_\mathrm{nl}(L,\delta^{x,1})-\epsilon_\mathrm{nl}(L,\delta^{x,2})\|_{0,1}
  &\leq\rho_\mathrm{nl}\|\delta^{x,1}-\delta^{x,2}\|_{0,1}\leq\frac{\kappa}{C_\mathrm{stab}}\|\delta^{x,1}-\delta^{x,2}\|_{0,1}         
\end{align*}
for $\delta^{x,1},\delta^{x,2}\in\mathcal{B}^{0,1}_r$ such that by definition of $C_\mathrm{stab}$
and Corollary \ref{thm:stabilityPhiL} and Lemma \ref{lemma:cons} we have
\begin{align*}
\|h(\delta^{x,1})-h(\delta^{x,2})\|_{0,1}\leq&\|[I-D\Phi_L(x^*)]^{-1}\|_{C^{0,1}_\mathrm{e}\leftarrow C^{0,1}_\mathrm{e}}\frac{\kappa}{C_\mathrm{stab}}\|\delta^{x,1}-\delta^{x,2}\|_{0,1}
\\
&\leq \kappa\|\delta^{x,1}-\delta^{x,2}\|_{0,1}.
\end{align*}
 This implies that $h$ has a unique fixed point in $B^{0,1}_r(0)$. Hence, by the identity \eqref{dfp2:split}, $\Phi_L$ has a unique fixed point in $B^{0,1}_r(x^*)$. Finally, inequality \eqref{convproof:h:bound} implies for the fixed point $\delta^{x,L}$ of $h$ that
$$
\|\delta^{x,L}\|_{0,1}\leq\frac{C_\mathrm{stab}}{1-\kappa} \|\epsilon_\mathrm{c}(L)\|_{0,1}=O(L^{-\min\{\lm,m\}}).
$$
\end{proof}
\paragraph{Precise bound of discretization error}
Revisiting the definition \eqref{def:Cstab} for the bound
$C_\mathrm{stab}$ in Corollary~\ref{thm:stabilityPhiL}, we observe
that we can replace $C_\mathrm{stab}$ by the smaller
$C_{\mathrm{stab},L}$, which can be chosen as close to the
infinite-dimensional stability constant $C_{\mathrm{stab},\infty}$ as
one wishes if one increases the discretization level $L$
further. Similarly, we may choose the constant $\kappa$ as close to
$0$ as desired, where again a smaller $\kappa$ requires a larger lower
bound on $L$. Furthermore, assuming that the order of mild
differentiability $\lm$ exceeds the order $m$ of the discretization,
we may insert the concrete estimate \eqref{disc:error:mleqlmax} for
$\epsilon_\mathrm{c}(L)$. We also observe that the $m$th time
derivatives $[g(x^*)]^{(m)}$ equal the derivatives of order $m+1$ of
$v^*$, when we denote the first component of the solution $x^*$ as
$v^*$ ($x^*=(v^*,\alpha^*,\mu^*)$). With these concrete estimates we
obtain that for every $\epsilon>0$ there exists a lower bound
$L_\mathrm{low}(\epsilon)$ such that the error of the discretized
fixed point problem $\delta^{x,L}$ satisfies
\begin{align}\label{disc:error:estimate}
  \|\delta^{x,L}\|_{0,1}\leq(1+\epsilon) C_{\mathrm{stab},\infty} \frac{\|\mathcal{L}\|_{C^{0,1}_\mathrm{e}\leftarrow
      L^\infty_\mathrm{e}}\|(v^*)^{(m+1)}\|_0}{m!}  \left[\frac{C_\mathrm{msh}}{L}\right]^m
\end{align}
for all $L\geq L_\mathrm{low}(\epsilon)$. This estimate is identical to the
result one would obtain for discretizations of smooth nonlinear
infinite-dimensional differential equations.

\section{Implementation and numerical test}
\label{sec:tests}
The numerical tool \textsc{DDE-Biftool} permits implementation of systems of FDEs with finitely many discrete delays, which have the form
\begin{align}
  \label{biftool:fde}
  M\dot y(t)&=f(y(t-\tau_0),\ldots,y(t-\tau_{n_\mathrm{d}}),p)\mbox{,}&&\mbox{where $\tau_0=0$, and }\\
  \tau_j&=\tau_{\mathrm{fun},j}(y(t-\tau_0),\ldots,y(t-\tau_{j-1}),p)&&\mbox{for $j=1,\ldots,n_\mathrm{d}$,}\nonumber
\end{align}
and $M\in \R^{n_y\times n_y}$, $y(t)\in\R^{n_y}$, and
$f:\R^{n_y\times (n_\mathrm{d}+1)}\times \R^{n_p}\to\R^{n_y}$,
$\tau_{\mathrm{fun},j}:\R^{n_y\times j}\times \R^{n_p}\to\R$ are
smooth functions of their arguments. The matrix $M$ may be singular to
permit the implicit definition of delays, or the formulation of
neutral FDEs (not analyzed in this paper).
For this class of FDEs the abstract right-hand side $G_\mathrm{FDE}$ in \eqref{intro:fde} has the form
\begin{align}
  \label{eq:biftool:gfde}
  G(y,p)&=f(y(-\tau_0),\ldots,y(-\tau_{n_\mathrm{d}}),p)\mbox{,}&&\mbox{where $\tau_0=0$, and }\\
  \tau_j&=\tau_{\mathrm{fun},j}(y(-\tau_0),\ldots,y(-\tau_{j-1}),p)&&\mbox{for $j=1,\ldots,n_\mathrm{d}$.}\nonumber
\end{align}
The functional $G_\mathrm{FDE}$ is mildly differentiable $\lm$ times
if the coefficient functions $f$ and $\tau_{\mathrm{fun},j}$ are $\lm$
times continuously differentiable with respect to their arguments. 

We perform our tests on the BVP \eqref{dde:ex1:rescaled}-\eqref{Raff:ex1} for $y_0=0.75$. As a starting guess for the Newton iterations we choose the solution computed with \textsc{DDE-Biftool}, unadapted mesh with $L=200$ and $m=7$. We recompute the solution using different values of $L$ and $m$ and approximate $\|y'/T+y(\cdot-p/T-y/T)\|_0$ by considering the maximum of the residuals on a uniform grid of $10001$ points, as shown in the left plot of Figure \ref{fig:sdproto_conv}. The picture on the right shows the (rescaled) solution obtained using $L=10$ and $m=5$, having actual period $T\approx 7.00$.

\begin{figure}
\centering
\begin{subfigure}[t]{0.52\textwidth}
    \includegraphics[scale=0.8]{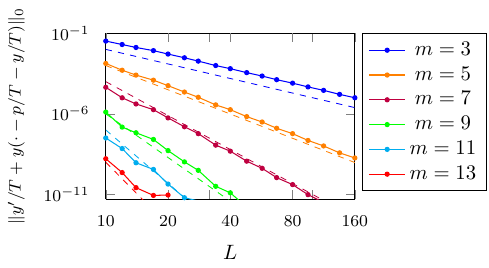}
\end{subfigure}%
~\begin{subfigure}[t]{0.45\textwidth}
\includegraphics[scale=.8]{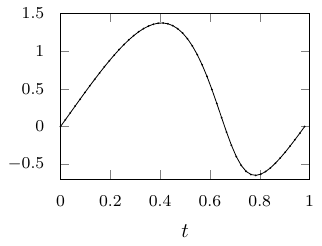}%
\end{subfigure}
\caption{Left: maximum among the errors computed at $10001$ equidistant points in $[0,T]$ for different values of $m$ and $L$, compared to straight dashed lines having slope $m$. Right: periodic solution computed with $L=10$ and $m=5$, rescaled so as to have period 1.}
\label{fig:sdproto_conv}
\end{figure}
\section{Conclusions}
In numerical bifurcation analysis of FDEs there is often a strong interest in analyzing the long-term dynamics. Such analysis includes the detection and computation of equilibria and periodic orbits regardless of their dynamical stability in a continuation framework with respect to model parameters.

The present paper provided a complete and rigorous error analysis of the piecewise orthogonal collocation for computing periodic solutions of FDEs which may include state-dependent delays. Indeed, although the method has widely been used for two decades \citep{ELR02,S06c} and incorporated into software such as \textsc{DDE-Biftool} and \textsc{knut}, the convergence of the corresponding finite-element method
had only been supported by (many) numerical experiments but never proved theoretically for a general FDE.

A convergence analysis was recently performed for FDEs with constant delays  \citep{andoSIAM2020}, based on the general approach for BVPs established in \citep{mas15NM}. The latter assumes regularity properties on the part of the right-hand sides that cannot be satisfied when state-dependent delays are present. However, using the concepts of weaker forms of differentiability developed for state-dependent delays \citep{cassidy2019,HKWW06}, allowed us to arrive at convergence estimates as sharp as those previously obtained for constant delays by only assuming to satisfy the regularity properties in the milder form. 

In numerical bifurcation analysis the next step after finding a periodic orbit numerically, is consideration of its linear stability, which ensures local asymptotic stability with respect to compatible perturbations in the initial value \citep{MP11}. For numerical computation of linear stability one needs to approximate the monodromy operator and the Floquet multipliers. Although we are not aware of any theoretical study on the convergence of such approximation with respect to the FEM strategy, preliminary experimental evidence can be found in \citep{blv22}.

The analysis carried out does not immediately extend to the spectral approach, characterized by bounded $L$ and increasing $m$, since it does assume in various points that the relevant interpolation operator is bounded. However, several of the apparent impediments to such an extension observed in \citep{andoSIAM2020} no longer hold once the BVP is formulated in periodic spaces of functions, as we have done in this paper. Therefore, the authors plan to reinvestigate the potential for spectral convergence  for periodic BVPs of FDEs. Moreover, there is potential for extending the analysis to more general classes of delay equations, such as neutral FDEs, whose further restrictions on the regularity on the right-hand sides represent a substantial obstacle.

Another issue was left open in \citep{andoSIAM2020}, even for the case of constant delays. Since the $y$-component of a numerical approximation is not differentiable for typical polynomial collocation schemes, it is unclear how the derivative of the right-hand side can be evaluated during Newton iterations, as it requires evaluation of $y'(t)$ at difficult-to-control times $t$. This discontinuous dependence on the solution implies that the standard
convergence argument for Newton iterations needs to be revisited for
FDEs.

\section*{Acknowledgments}
A.\,A.\ is a member of INdAM Research group GNCS, as well as of UMI Research group “Modellistica socio-epidemiologica”. This work was supported by the Italian Ministry of University and Research (MUR) through the PRIN 2020 project (No. 2020JLWP23) “Integrated Mathematical Approaches to Socio–Epidemiological Dynamics”, Unit of Udine (CUP G25F22000430006). The research collaboration was supported by the Lorentz Center Leiden (The Netherlands) workshop ``\emph{Towards rigorous results in state-dependent delay equations}'',
4--8 March 2024.
\bibliographystyle{abbrvnat}
\bibliography{delay}
\appendix
\section{Proof of Lemma~\ref{thm:diff:shift}}
\label{sec:proof:diff:shift}
The extendability condition \ref{def:mild:ext} in
Definition~\ref{def:mild} is only formulated for the application of
$\ell$ times the same deviation $\delta^u$. The polarization identity
(Proposition~\ref{thm:pol} below) ensures that condition \ref{def:mild:ext}
also applies to $\ell$ different deviations. 
\begin{lemma}[Multidirectional extendability]\label{thm:mild:ext0}
  If $G$ is $\lm$ times mildly differentiable according to
  Definition~\ref{def:mild}, then the map
  \begin{align*} C^\ell\times
    (C^\ell)^\ell\ni (x,y^1,\ldots,y^\ell)\mapsto D^\ell
    G(x)y^1\ldots y^\ell\in \R^{n_G}
  \end{align*}
  can be extended to a continuous map in
  $C^\ell\times (C^{\ell-1})^\ell$ for all $\ell\in\{1,\ldots,\lm\}$. 
\end{lemma}
\begin{proposition}[Polarization identity]\label{thm:pol}
  Let $\ell\geq1$ be arbitrary.  There exist $2^\ell$
  coefficients $a_i\in\R$ and $\ell\, 2^{\ell-1}$ coefficients
  $b_{i,j}\in\{-1,1\}$ \textup{(}$i\in\{1,\ldots,2^{\ell-1}\}$,
  $j\in\{1,\ldots,\ell\}$\textup{)} such that for all $c_1,\ldots,c_\ell\in\R$
  the following identity holds:
\begin{align*}
  \prod_{j=1}^\ell c_i=\sum_{i=1}^{2^{\ell-1}}a_i\left[\sum_{j=1}^\ell b_{i,j}c_j\right]^\ell.
\end{align*}
\end{proposition}
Consequently, for any bounded $\ell$-linear map $M$ from
$X\times\ldots\times X$ to $Y$ and arbitrary linear spaces $X$ and $Y$
  $Mx_1\ldots x_\ell=\sum_{i=1}^{2^{\ell-1}}a_iM\left[\sum_{j=1}^\ell b_{i,j}x_j\right]^\ell$ for all $x_1,\ldots,x_\ell\in X$.
In short, arbitrary combinations of arguments of the multilinear map
can be expressed as linear combination of the single-argument map
$y\mapsto M y^\ell$. For example, for $\ell=2$, the coefficients are
$a_1=1/4$, $a_2=-1/4$, $b_{1,1}=b_{2,1}=b_{1,2}=1$, $b_{2,2}=-1$:
$Myz=\frac{1}{4}M[y+z]^2-\frac{1}{4}M[y-z]^2$. Lemma~\ref{thm:mild:ext0} follows by continuity.

\paragraph{Proof of Lemma~\ref{thm:diff:shift}} We show  inductively over $k$ that
\begin{align*}
C^{k+\ell}_\pi\times (C^{k+\ell-1}_\pi)^\ell\ni(x,y^1\!\!,\ldots,y^\ell)\mapsto D^\ell\gts(x)y^1\!\!\ldots y^\ell=\left[t\mapsto  D^\ell G(x_t)y^1_t\ldots y^\ell_t\right]\in C^k_\pi
\end{align*}
is continuous, if $G$ is $k+\ell$ times mildly differentiable. The
statement of Lemma~\ref{thm:diff:shift} then follows from setting $k=\lm-\ell$.

For $k=0$ the fact that $D^\ell\gts\vert_{C^\ell_\pi}$ is well defined
at each time $t$ follows from the assumption that $G\vert_{C^\ell_\pi}$ is $\ell$ times
continuously differentiable by part~\ref{def:mild:rest} of Definition~\ref{def:mild}. The continuity with respect to
the $C^0_\pi$-norm in $t$ follows from the fact that the time shift 
  $\R\times C^j_\pi\ni(t,u)\mapsto u_t\in C^j_\pi$
is continuous for arbitrary $j\geq0$.

For the inductive step, we first assume that for all integers $\ell\geq0$ the function
$t\mapsto D^\ell G(x_t)y^1_t\ldots y^\ell_t$ is in $C^k_\pi$ for all
$x\in C^{k+\ell}_\pi$ and all $y^1,\ldots,y^\ell\in C^{k+\ell-1}_\pi$ if $G$ is $k+\ell$ times mildly differentiable. We have to show that
$D^\ell G(x_t)y^1_t\ldots y^\ell_t$ is in $C^{k+1}_\pi$ for
$x\in C^{k+\ell+1}_\pi$, $y^1,\ldots,y^\ell\in C^{k+\ell}_\pi$ if $G$ is $k+\ell+1$ times mildly differentiable. The time
difference quotient for $D^\ell G(x_t)y^1_t\ldots y^\ell_t$ at time $t$ equals
\begin{align}
  \frac{1}{h}\left[D^\ell G(x_{t+h})\prod_{i=1}^\ell y^i_t-D^\ell G(x_t)\prod_{i=1}^\ell y^i_t\right]+
\sum_{j=1}^\ell D^\ell G(x_t)\left[\prod_{i\neq j}^\ell y^i_t\right]\frac{y^j_{t+h}-y^j_t}{h}
\label{proof:gdiff:gyth}
\end{align}
The function $u\mapsto D^\ell G(u)\prod_{i=1}^\ell y^i_t$ is
continuously differentiable with respect to $u$ in
$u=sx_{t+h}+(1-s)x_t$ for $s\in[0,1]$ since $x\in C^{k+\ell+1}_\pi$,
$y^1,\ldots,y^k\in C^{k+\ell}_\pi$ and $k\geq1$ by the assumption
that $G$ is $k+\ell+1$ times mildly differentiable and
$u\in C^{k+\ell+1}_\pi$ and $y^1,\ldots,y^\ell\in C^{k+\ell}_\pi$ with
$k+\ell\geq\ell+1$. Thus, we may apply the mean-value theorem for the first term
in \eqref{proof:gdiff:gyth}, resulting in the expression
\begin{align*}
\int_0^1D^{\ell+1}G(sx_{t+h}+(1-s)x_t)\frac{x_{t+h}-x_t}{h}\prod_{i=1}^\ell y^i_t\dd s
\end{align*}
By Lemma~\ref{thm:mild:ext0} the derivative
$D^{\ell+1}G(u)v_{0}\ldots v_j$ is continuous for arguments
$u\in C^{\ell+1}_\pi$, $v_0,\ldots,v_\ell\in C^\ell_\pi$. The argument
$u=sx_{t+h}+(1-s)x_t$ is in $C^{k+\ell+1}_\pi$, $v_0=[x_{t+h}-x_t]/h$ is in $C^{k+\ell+1}_\pi$ and $ y^i_t$ and the
limit $x_t'$ of $[x_{t+h}-x_t]/h$ are in $C^{k+\ell}_\pi$ such that we can
take the limit for $h\to 0$ to obtain
\begin{align*}
  \lim_{h\to0}\frac{1}{h}\left[D^\ell G(x_{t+h})\prod_{i=1}^\ell y^i_t-D^\ell G(x_t)\prod_{i=1}^\ell y^i_t\right]=D^{\ell+1}G(x_t)x_t'\prod_{i=1}^\ell y^i_t.
\end{align*}
The right-hand side is a $C^k_\pi$ function by the assumption
of the inductive step. 

In the second sum in \eqref{proof:gdiff:gyth} the difference quotient
$[y^j_{t+h}-y^j_t]/h$ has the limit $(y^j_t)'$ in $C^{k+\ell-1}_\pi$
and the derivative $D^\ell G(x_t)$ can be extended continuously to
multilinear arguments in $C^{k+\ell-1}_\pi$, since $G$ is
$k+\ell\geq \ell$ times mildly differentiable. Thus, the limit for
$h\to0$ also exists for the second sum in \eqref{proof:gdiff:gyth}:
\begin{align}\label{proof:gdiff:ythlim}
  \lim_{h\to0}D^\ell G(x_t)\prod_{i\neq j}^\ell y^i_t\frac{y^j_{t+h}-y^j_t}{h}=D^\ell G(x_t)\prod_{i\neq j} y^i_t(y^j_t)'.
\end{align}
The argument $x$ in this limit is in $C^{k+\ell+1}_\pi$, the arguments
$y^i_t$ are in $C^{k+\ell}_\pi$ and $(y^j)'$ is in
$C^{k+\ell-1}_\pi$. Thus, by the assumption in the inductive step the
limits in \eqref{proof:gdiff:ythlim} are also in $C^k_\pi$ for
$j=1\ldots,\ell$. This implies that the limits of the both terms in \eqref{proof:gdiff:gyth}, and, thus, the limit of the time difference
quotient \eqref{proof:gdiff:gyth} are in $C^k_\pi$. Hence, $t\mapsto D^\ell G(x_t)y^1_t\ldots y^\ell_t$ is in $C^{k+1}_\pi$.
Finally, we check the continuity of the map
  $C^{k+\ell+1}_\pi\times (C^{k+\ell}_\pi)^\ell\ni (x,y^1,\ldots,y^\ell)\mapsto D^\ell\gts(x)y^1\ldots y^\ell\in C^{k+1}_\pi$.
By inductive assumption the map is continuous as a map into
$C^k_\pi$. Inspecting \eqref{proof:gdiff:gyth}
and \eqref{proof:gdiff:ythlim}, its derivative  in $t$ is
\begin{align*}
  D^{\ell+1}G(x_t)x_t'\prod_{i=1}^\ell y^i_t+\sum_{j=1}^\ell D^\ell G(x_t)\prod_{i\neq j}^\ell y^i_t(y^j_t)'.
\end{align*}
All terms in this sum are continuously mapping $C^{k+\ell+1}_\pi\times (C^{k+\ell}_\pi)^\ell\ni (x,y^1,\ldots,y^\ell)$ into $C^k_\pi$ by inductive
assumption, such that $(x,y^1,\ldots,y^\ell)\mapsto D^\ell\gts(x)y^1\ldots y^\ell$ is continuous into $C^{k+1}_\pi$.
\hfill\mbox{(end of proof for Lemma~\ref{thm:diff:shift})}
\end{document}

\begin{equation}\label{}
\left\{
\begin{aligned}
A &= (-1)^k \frac{4}{k\pi} \sin \left(\frac{k \pi}{2}\right) \frac{\gamma}{2} (1-2\theta) A, \\
\theta &= \gamma \theta (1-\theta) -\frac{\gamma}{2}A^{2}.
\end{aligned}
\right.
\end{equation}

\begin{align}
\mathcal{K}(\alpha_{1},\alpha_{2})&\coloneqq e^{\int_{\alpha_{2}}^{\alpha_{1}}g_{1}(\theta)d\theta}g_{2}(\alpha_{2}),\label{K}\\
\mathcal{Klambda}(\alpha_{1},\alpha_{2})&\coloneqq -\mathcal{F}(\alpha_{1},\bar{TB})\left(\int_{\alpha_{2}}^{\alpha_{1}}\mu_{1}(\theta)\mathcal{K}(\theta,\alpha_{2})d\theta+\mu_{2}(\alpha_{2})\right).\label{Klambda}
\end{align}

\begin{equation*}
\begin{split}
z(t+4)={}& \int_{0}^4 A(\tau) h(z(t+4-\tau)) {\rm d} \tau \\
={}& \int_{0}^4 A(4-\sigma)h(z(t+\nu)){\rm d} \sigma \\
={}& \int_{0}^4 A(\tau) h(z(t+\tau)) {\rm d} \tau \\
={}& \int_{0}^4 A(\tau) h(z(-t-\tau)) {\rm d} \tau \\
={}& z(-t),
\end{split}
\end{equation*}

\begin{equation}\label{equation1}
\begin{split}
x'&=x(\alpha-\beta y),\\
y'&=-y(\gamma-\delta x).
\end{split}
\end{equation}

\begin{equation}\label{equation2}
\begin{split}
\dot{y}(t) ={}& a y(t)+b f(y(t-\gamma_{0}\tau)))\\
>{}& a y(t)-b f(y(t-\gamma_{0}\tau))\\
& +b f(y(t-\gamma_{0}\tau))\\
={}& a y(t).
\end{split}
\end{equation}

\begin{align}
\dot{y}(t) &= a y(t)+b f(y(t)),\label{equation3a}\\
\ddot{y}(t) &= a y(t) - b f(y(t)),\label{equation3b}\\
y(0) &= y_{0}.\label{equation3c}
\end{align}

\begin{equation}\label{equation4}
|x|=\begin{cases}
x,\quad &x \geq 0,\\
-x, \quad &x<0.
\end{cases}
\end{equation}

\begin{equation*}
\begin{split}
\alpha x(t)+\beta y(t)= {} &\int_{t_{0}}^{t}f(x(s),y(s-\tau))\,ds-\int_{t_{0}}^{t}g(x(s-\tau),y(s))\,ds\\
&+\int_{t_{0}}^{t}h(x(s-\tau),y(s-\tau))\,ds.
\end{split}
\end{equation*}

\begin{figure}[!h]
\centering
\includegraphics[width=\textwidth]{}
\caption{caption. See text for more details.}\label{f_E}
\end{figure}

\section{}
\label{s_}
\subsection{}
\label{s_}

{\color{red} ()}
